\documentclass[11pt, leqno, twoside]{article}

\usepackage{amssymb}
\usepackage{amsmath}
\usepackage{amsthm}
\usepackage{amsfonts}
\usepackage{mathrsfs}
\usepackage{indentfirst}
\usepackage{graphicx}
\usepackage{txfonts}

\usepackage{anysize}

\usepackage{fancyhdr}

\usepackage{color}
\usepackage{setspace}

\usepackage{hyperref}
\hypersetup{colorlinks=true, linkcolor=blue, citecolor=blue}

\allowdisplaybreaks

\usepackage{txfonts}

\newtheorem{thm}{Theorem}[section]
\newtheorem{prop}[thm]{Proposition}
\newtheorem{lem}[thm]{Lemma}

\theoremstyle{definition}
\newtheorem{defn}[thm]{Definition}

\numberwithin{equation}{section}

\newcommand{\rt}{\right}

\def\hs{\hspace{0.26cm}}
\def\ls{\lesssim}
\begin{document}
	
	\title{\bf\Large A Note on Capacity Muckenhoupt Weights with respect to the Hausdorff Content
		\footnotetext{\hspace{-0.35cm}
			2020 {\it Mathematics Subject Classification}.
			{42B25} \endgraf
			{\it Key words and phrases.} Muckenhoupt weight, Hausdorff measure, Hausdorff dimension}}
\author{Yangzhi Zhang,
Ciqiang Zhuo\footnote{Corresponding author,\ E-mail: cqzhuo87@hunnu.edu.cn/}}
	
	\date{ }
	\maketitle
	
	\vspace{-0.6cm}
	
	\begin{center}
		\begin{minipage}{13cm}
			{\small {\bf Abstract}\quad
				Let $p\in[1,\infty)$ and $\delta\in(0,n]$.
				Recently, to characterize the weighted boundedness of maximal operator on Choquet integrals based on Hausdorff content $\mathcal H_\infty^\delta$, a class of Capacity Muckenhoupt weights is introduced, denoted by $\mathcal A_{p,\delta}$, and,
				for any fixed $\delta\in(0,n]$,
				this new class of weights is proved to satisfy the self-improving property with respect to the index $p\in(1,\infty)$.
				In this note, we show that, for every fixed $p\in[1,\infty)$,
				the class of capacity Muckenhoupt weights $\mathcal A_{p,\delta}$ fails to satisfy the self-improving property with respect to the index $\delta\in(0,n]$.}
		\end{minipage}
	\end{center}
	
	\vspace{0.2cm}

\def\vpz{\vphantom}
\def\rr{{\mathbb R}}
\def\rn{{{\rr}^n}}
\def\rnn{{\rr}^{n+1}_+}
\def\zz{{\mathbb Z}}
\def\cc{{\mathbb C}}
\def\nn{{\mathbb N}}
\def\cp{{\mathcal P}}
\def\cq{{\mathcal Q}}
\def\cf{{\mathcal F}}
\def\cl{{\mathcal L}}
\def\cm{{\mathcal M}}
\def\cn{{\mathcal N}}
\def\mj{{\mathrm J}}
\def\ca{{\mathcal A}}
\def\cd{{\mathcal D}}
\def\cb{{\mathcal B}}
\def\ch{{\mathcal H}}
\def\ck{{\mathcal K}}

\def\cha{{\mathbf{1}}}

\def\fz{\infty}
\def\az{\alpha}
\def\bz{\beta}
\def\dz{\delta}
\def\bdz{\Delta}
\def\ez{\epsilon}
\def\gz{{\gamma}}
\def\bgz{{\Gamma}}
\def\kz{\kappa}
\def\lz{\lambda}
\def\blz{\Lambda}
\def\oz{{\omega}}
\def\boz{{\Omega}}
\def\tz{\theta}
\def\sz{\sigma}

\def\bfai{\Phi}
\def\ccr{{\mathcal R}}
\def\epz{\epsilon}
\def\uz{\upsilon}
\def\rz{\rho}
\def\xz{\xi}
\def\zez{\zeta}
\def\vz{\varphi}
\def\uc{{\varepsilon}}

\def\lf{\left}
\def\rt{\right}
\def\lfz{{\lfloor}}
\def\rfz{{\rfloor}}
\def\la{\langle}
\def\ra{\rangle}
\def\hs{\hspace{0.26cm}}
\def\hsz{\hspace{-0.26cm}}
\def\ls{\lesssim}
\def\gs{\gtrsim}
\def\tr{\triangle}
\def\pa{\partial}
\def\ov{\overline}
\def\noz{\nonumber}
\def\wz{\widetilde}
\def\wh{\widehat}
\def\st{\subset}
\def\com{\complement}
\def\bh{\backslash}
\def\btd{\bigtriangledown}

\def\cs{{\mathcal S}}
\def\cx{{\mathcal X}}

\def\gfz{\genfrac{}{}{0pt}{}}
\def\dyt{\,\frac{dy\,dt}{t}}
\def\dxt{\,\frac{dx\,dt}{t}}
\def\dxtn{\,\frac{dx\,dt}{t^{n+1}}}
\def\dt{\,\frac{dt}{t}}
\def\dytn{\,\frac{dy\,dt}{t^{n+1}}}
\def\dtn{\,\frac{dt}{t^{n+1}}}

\def\dis{\displaystyle}
\def\dlimsup{\displaystyle\limsup}
\def\dist{\mathop\mathrm{\,dist\,}}
\def\supp{\mathop\mathrm{\,supp\,}}
\def\atom{\mathop\mathrm{\,atom\,}}
\def\loc{{\mathop\mathrm{\,loc\,}}}
\def\esinf{\mathop\mathrm{\,ess\,inf\,}}
\def\esup{\mathop\mathrm{\,ess\,sup\,}}
\def\fin{{\mathop\mathrm{fin}}}
\def\Lip{{\mathop\mathrm{\,Lip\,}}}

\def\aa{{\mathbb A}}
\def\q1{\wz q}
\def\Q1{q_1}
\def\vlp{{L^{p(\cdot)}(\cx)}}
\def\vlq{{L^{q(\cdot)}(\cx)}}
\def\vlpq{{L^{p(\cdot)}(Q)}}
\def\vcm{{\mathcal L_{p(\cdot),\phi,d}(\rn)}}
\def\aaa{{\ca_{p(\cdot)}^s(\cx)}}
\def\bbb{{\cb_{p(\cdot)}^s(\cx)}}
\def\www{{\cw_{p(\cdot)}^s(\cx)}}

\def\dis{\displaystyle}
\def\dsum{\displaystyle\sum}
\def\dint{\displaystyle\int}
\def\dfrac{\displaystyle\frac}
\def\dsup{\displaystyle\sup}
\def\dinf{\displaystyle\inf}
\def\dlim{\displaystyle\lim}

\def\diam{{\mathop\mathrm{\,diam\,}}}
\def\dist{{\mathop\mathrm{\,dist\,}}}
\def\loc{{\mathop\mathrm{loc\,}}}
\def\lfz{\lfloor}
\def\rfz{\rfloor}
\def\wee{\wedge}

\section{Introduction}
Let $\dz \in (0,n]$ and $\ch_\fz^\delta$ be
the \emph{Hausdorff content} (or \emph{$\delta$-dimensional Hausdorff capacity}) of a subset $E\subset \rn$ (see Section 2
for the definition).
For any subset $E\subset \rn$ and any non-negative function $f$ defined on $E$,
the \emph{Choquet integral} of $f$ on $E$ is defined by setting
\[\int_{E}f(x)\,d\ch_{\fz}^{\dz}:=\int_{0}^{\fz}\ch_{\fz}^{\dz}(\{x\in E:\ f(x)>t\})\,dt.\]
For any $p\in(0,\fz)$, the \emph{$p$-Choquet
	integral with respect to the capacity} of a function $f$ on $E$ is defined by setting
\begin{align*}
	\|f\|_{L^p(E,\ch_{\fz}^{\dz})}:=\lf[\int_E |f(x)|^p\,d\ch_{\fz}^{\dz}\rt]^\frac1p.
\end{align*}
For any $p\in(0,\fz)$ and $E\subset \rn$, we always use $L^p(E,\ch_{\fz}^{\dz})$ to denote
the space of all functions $f$ such that the quasi-norm $\|f\|_{L^p(E,\ch_{\fz}^{\dz})}$ is finite, and denote by $L_{\loc}^1(\rn,\ch_{\fz}^{\dz})$ the set of all functions $f$ satisfying that $f\in L^1(K,\ch_{\fz}^{\dz})$ for any compact set $K\subset\rn$. For the further theory of Choquet integrals with respect to Hausdorff content, we refer the reader to \cite{0923-2,Ada15,OrVe98,Ta11}.

Recently, inspired by the recent extensive research on capacity (see, for example, \cite{chyz26,flz25,JLWX26}), Huang et al. \cite{0702-4} introduced a class of capacitary Muckenhoupt weight with respect to the Hausdorff content.
Precisely, for $p\in[1,\fz)$, the \emph{capacity Muckenhoupt weight} $\ca_{p,\dz}$ is defined to be the set of
all $\ch_\fz^\delta$-capacity weight on $\rn$ such that $[w]_{\mathcal A_{p,\delta}}<\infty$,
where for $p\in(1,\fz)$,
\begin{align*}
	[w]_{\mathcal A_{p,\delta}}:=\sup_{\,Q\,\subset \rn}\lf\{\frac{1}{\mathcal H^{\delta}_\infty(Q)}\int_Qw(x)\,d\mathcal H^{\delta}_\infty\rt\}
	\lf\{\frac{1}{\mathcal H^{\delta}_\infty(Q)}\int_Qw(x)^{-\frac{1}{p-1}}\,d\mathcal H^{\delta}_\infty\rt\}^{p-1}
\end{align*}
with the supremum being taken over all cubes $Q$ of $\rn$, and for $p=1$,
\begin{align*}
	[w]_{\ca_{1,\delta}}:=\inf\lf\{K:\ \cm_{\ch_\fz^\delta}w\leq K w\ {\rm for}\ \ch_\fz^\delta{\rm-almost\ everywhere}\rt\}.
\end{align*}
Here, a function $w$ is called a $\ch_\fz^\delta$-\emph{capacity weight} on $\rn$, we mean that
$w\in L_{\loc}^1(\rn,\ch_{\fz}^{\dz})$ and $w\in(0,\fz)$ almost everywhere on $\rn$
with respect to the Hausdorff content $\ch_\fz^\delta$, and \(\mathcal M_{\mathcal H_{\infty}^\delta}\) stands for the Hardy-Littlewood maximal operator associated with the Hausdorff content defined by setting
$$\mathcal M_{{\mathcal H}_\infty^\delta}f(x):=\sup_{Q\ni x}\frac1{{\mathcal H}_\infty^\delta(Q)}\int_Q |f(y)|\,d{\mathcal H}_\infty^\delta,$$
with the supremum ranging over all cubes \(Q\) containing \(x\).

When $\delta=n$, the capacity Muckenhoupt weight $\mathcal A_{p,n}$, restricted on all Lebesgue measurable subsets
of $\rn$, is equivalent to the classical Muckenhoupt weight $A_p$
introduced in \cite{Muc72} and have several interesting applications.
For example, such capacity Muckenhoupt weight gives a sufficiency and necessity for the boundedness of the maximal operator
\(\mathcal M_{\mathcal H_{\infty}^\delta}\) on $L^p(\ch_\fz^\delta)$ (see \cite[Theorems 1.1 and 1.2]{0702-4}).
As an analogue of the classical BMO space, functions with $\delta$-dimensional bounded mean oscillation
also enjoy an equivalent characterization in terms of the capacity Muckenhoupt weight (see \cite{HZZ26}).

Besides the above interesting application, Huang et al. in \cite{0702-4} further established
several properties of the class of capacity Muckenhoupt weights $\ca_{p,\delta}$, such as the reverse H\"older inequality,
the Jones factorization theorem, and monotonicity with respect to $p$.
Particularly, for a given $\delta\in(0,n]$, $\ca_{p,\delta}$ has the self-improving property with respect to the index $p$,
namely, if $w\in\mathcal{A}_{p,\delta}$ with $p\in(1,\fz)$, then there exists $q\in(1,p)$ such that $w\in\mathcal{A}_{q,\delta}$.
On the other hand, by \cite[Proposition 2.15]{0702-4}, we know that, for any fixed $p\in[1,\fz)$, if $0<\delta<\beta\le n$,
then $\mathcal{A}_{p,\delta}\subsetneqq \mathcal{A}_{p,\beta}$.
Therefore, a natural interesting question arises:

	Does the class of capacity Muckenhoupt weights
	have the self-improving property with respect to the dimension $\delta$? That is, for any given $p\in[1,\infty)$ and $w\in\mathcal{A}_{p,\delta}$ with $\delta\in(0,n]$, does there exist $\delta_1\in(0,\delta)$ such that $w\in\mathcal{A}_{p,\delta_1}?$

The main result of this note gives a negative answer to the above question.

\begin{thm}\label{lm-1226-x}
	Let $\dz\in(0, n]$ and $p\in[1,\fz)$. Then there exists a capacity weight \(w\) such that \(w\in\mathcal{A}_{p,\delta}\) but \(w\notin\mathcal{A}_{p,\beta}\) for any \(\beta\in(0,\delta)\), namely, the capacity weights \(\mathcal{A}_{p,\delta}\) do not have the self-improving property with respect to the index \(\delta\).
\end{thm}

The proof of Theorem \ref{lm-1226-x} is constructive, which relies on the critical properties of Hausdorff dimension.
For any set $E\subset\rn$, the Hausdorff dimension ${\rm dim}_\ch(E)$ and the $s$-dimensional Hausdorff measure $\ch^s(E)$ satisfy
\[
\ch^s(E)=\lf\{
\begin{array}{ll}
	\fz,&\quad 0\le s<{\rm dim}_\ch(E),\\
	0,&\quad s>{\rm dim}_\ch(E),
\end{array}
\rt.
\]
see \cite[(3.11)]{f14}. Here, for any subset $E$, $s={\rm dim}_\ch(E)$ is the critical value
for the Hausdorff measure $\ch^s(E)$ from $\infty$ to $0$.
Moreover, it was pointed out in \cite{f14}
that the Hausdorff measure at the critical value would be zero, a finite positive number or infinity.
However, it seems that there is no specific example satisfying
\begin{equation}\label{eq0317-b}
	{\rm dim}_\ch(E)=s \quad {\rm and}\quad \ch^s(E)=0.
\end{equation}
In this note, we construct a specific subset $E\subset\rn$ satisfying \eqref{eq0317-b}, through a method from the iterative construction of generalized Cantor sets (see \cite{f14}),
by adjusting the sparsity of the Cantor set and the decay rate of the interval lengths.
Using this example, we obtain a weight $w$ satisfying the requirement of Theorem \ref{lm-1226-x}.

\section{Proof of Theorem \ref{lm-1226-x}} \label{s2}

We first recall the definitions of Hausdorff measure and Hausdorff dimension (see \cite{eg92,f14}).

\begin{defn}
	Let $\delta\in(0,n]$. For any $\epsilon\in(0,\fz]$ and any subset $E\subset\mathbb{R}^n$,
	the \emph{$\delta$-dimensional $\epsilon$-Hausdorff outer measure} of $E$ is defined by setting
	\begin{align*}
		\mathcal{H}_{\epsilon}^{\delta }\lf (E \rt )
		:=\inf\lf \{\sum_i [l(Q_i)]^{\delta }:\
		E\subset \bigcup_i Q_i, l(Q_i)<\epsilon \right \},
	\end{align*}
	where the infimum is taken over all finite or countable cubes from $\rn$ coverings
	$\{Q_i\}_{i}$ of $E$ and $l(Q)$ denotes the edge length
	of the cube $Q$.
	
	Particularly, $\ch_\fz^\delta(E)$ is called the \emph{Hausdorff content} (or \emph{$\delta$-dimensional Hausdorff capacity})
	of $E$.
	Furthermore, the \emph{$\delta$-dimensional Hausdorff measure} of $E$ is defined as
	\begin{equation}\label{eq0319-a}
		\mathcal{H}^{\delta}(E) := \lim_{\epsilon\to 0^+} \mathcal{H}_\epsilon^{\delta}(E).
	\end{equation}
\end{defn}

We remark that the limit in \eqref{eq0319-a} exists for any subset $E\subset \rn$, although the limiting value can be (and
usually is) zero or infinity. Moreover,
$$\mathcal{H}^{\delta}(E) = \sup_{\epsilon> 0} \mathcal{H}_\epsilon^{\delta}(E).$$

\begin{defn}
	Let $E\subset\mathbb{R}^n$ with $E\neq\emptyset$. Then the \emph{Hausdorff dimension} of $E$ is defined by
	\[
	\dim_{\mathcal{H}} (E) := \inf\left\{s\ge0: \mathcal{H}^{s}(E) = 0 \right\}
	\]
	In particular, the Hausdorff dimension of the empty set is defined to be $0$.
\end{defn}

We point out that, for any set $E$, $\dim_{\mathcal{H}} (E)$ is also the critical value of $s$ at which $\mathcal H^s(E)$ ``jumps" from $\infty$ to $0$ (see \cite[p.\,48]{f14}).
Alternative, when $s<\dim_{\mathcal{H}} (E)$, $\ch^s(E)=\fz$ and, when $s>\dim_{\mathcal{H}} (E)$, $\ch^s(E)=0$.
Thus,
$$\dim_{\mathcal{H}} (E)= \sup\left\{ \delta\ge0: \mathcal{H}^{\delta}(E) = +\infty \right\}.$$
When $\dim_{\mathcal{H}} (E)=s$, $\ch^s(E)$ may be zero, a finite positive number or infinity.
However, it seems that there is no specific example for the case ``zero" or ``infinity" to our knowledge.
The following conclusion is the main step for proof of Theorem \ref{lm-1226-x}, which confirms the possibility that,
for any $s\in(0,n]$, when $\dim_{\mathcal{H}} (E)=s$, $\ch^s(E)=0$.

\begin{prop}\label{them1229-1}
	Let $\delta\in(0,n]$. Then there exists a subset $E\subset\mathbb{R}^n$ such that
	\[
	\dim_{\mathcal{H}} (E)=\delta \quad \text{and} \quad \mathcal{H}^\delta(E) = 0.
	\]
\end{prop}

To prove Proposition \ref{them1229-1}, we need the following mass distribution principle (see \cite[p.\,67]{f14}).
Let $\mu$ be a measure on $\mathbb{R}^n$. The smallest closed set $F\subset \rn$ satisfying $\mu(\mathbb{R}^n\backslash F)=0$ is called the support of $\mu$, denoted by ${\rm{spt}}\,\mu$. Moreover, for any $F$ a subset of $\mathbb{R}^n$,
if ${\rm{spt}}\ \mu\subset F$ and $0<\mu(F)<\infty$, then $\mu$ is called a mass distribution on $F$.

\begin{lem}[Mass Distribution Principle]\label{them1230-1}
	Let $s\in(0,\fz)$, $F$ be a subset of $\mathbb{R}^n$ and $\mu$ a mass distribution on $F$.
	If there exist positive constants $K$ and
	$\epsilon$ such that, for every set $U\subset \rn$ with $\text{diam}\ U\leq \epsilon$,
	\[
	\mu(U)\leq K(\text{diam}\ U)^s,
	\]
	then
	\[
	\mathcal{H}^{s}(F)\ge \frac{\mu(F)}{K} \quad \text{and} \quad \dim_{\mathcal{H}} (F)\ge s.
	\]
\end{lem}

We are now ready to prove Proposition \ref{them1229-1}.

\begin{proof}[Proof of Proposition \ref{them1229-1}]
	
	We only give the proof of the case $n=1$ and $\delta\in(0,1]$,
	since it is not difficulty to extend to the general $n$.
	
	\textbf{Step 1: The Construction of $E$}. The subset $E$ is constructed by induction.
	For the initial step $(k=0)$, let $I_{0,1}=[0,1]$.
	For \(k\in\nn\), we take operation on closed intervals $I_{k-1,i}$ $(i=1,2,\dots,N_{k-1})$ from the $(k-1)$-th iteration as follows, where $N_k$ denotes the number of all intervals in the $k$-step ($N_0=1$).
	\begin{enumerate}
		\item[(i)] Select two closed subintervals at the endpoints of the closed interval \(I_{k-1,i}\) with length
		\[
		L_k := 2^{-\frac{k}{\delta} - \frac{k}{\ln (k+1)}},
		\]
		and remove the middle unselected part;
		\item[(ii)] Denote all selected closed intervals in the \(k\)-th iteration by \(\{I_{k,i}\}_{i=1}^{N_k}\), where the number of intervals satisfies
		\[
		N_k = 2\cdot N_{k-1} = 2^k.
		\]
	\end{enumerate}
	Let \(E_k := \bigcup_{i=1}^{N_k} I_{k,i}\). Then the sparse generalized Cantor set is defined as the intersection of the retained intervals from all iterations:
	\[
	E := \bigcap_{k=1}^{\infty} E_k.
	\]
	By the nested closed intervals theorem, \(E\neq\emptyset\) and \(E\subset[0,1]\subset\mathbb{R}\).
	
	\textbf{Step 2: Verification for \(\mathcal{H}^\delta(E) = 0\)}.
	For any \(k\ge1\), let \(\epsilon_k := L_k\) (the diameter of the intervals at the \(k\)-th iteration). Then \(\{I_{k,i}\}_{i=1}^{N_k}\) forms an \(\epsilon_k\)-cover of \(E\) and hence
	\[
	\mathcal{H}_{\epsilon_k}^\delta(E) \le \sum_{i=1}^{N_k} (\text{diam } I_{k,i})^\delta
	= 2^k \cdot \left( 2^{-\frac{k}{\delta} - \frac{k}{\ln (k+1)}} \right)^\delta
	= 2^{-\frac{k\delta}{\ln (k+1)}}.
	\]
	As \(k\to\infty\), \(\epsilon_k = L_k \to 0^+\) and \(\frac{k\delta}{\ln (k+1)} \to +\infty\). Therefore,
	\[
	\mathcal{H}^\delta(E) = \lim_{k\to\infty} \mathcal{H}_{\epsilon_k}^\delta(E)
	\le \lim_{k\to\infty} 2^{-\frac{k\delta}{\ln (k+1)}} = 0,
	\]
	which implies \(\mathcal{H}^\delta(E) = 0\).
	
	\textbf{Step 3: Proof for \(\dim_{\mathcal{H}}(E) = \delta\)}.
	By the definition of Hausdorff dimension, it is enough to show the following two statements:
	\begin{enumerate}
		\item[(1)] For any \(s>\delta\), \(\mathcal{H}^s(E)=0\);
		\item[(2)] For any \(s<\delta\), \(\dim_{\mathcal{H}} (E)\ge s\).
	\end{enumerate}
	
	For (1), we first have
	\begin{equation}\label{eq0317-a}
		\sum_{i=1}^{N_k} (\text{diam } I_{k,i})^s
		= 2^k \cdot 2^{-\frac{ks}{\delta} - \frac{ks}{\ln (k+1)}}
		= 2^{k(1 - \frac{s}{\delta}t) - \frac{ks}{\ln (k+1)}}.
	\end{equation}
	Since \(1 - \frac{s}{\delta} < 0\) by $s>\delta$, it follows that the right-hand side of \eqref{eq0317-a} tends to \(0\) as \(k\to\infty\).
	Thus, \(\mathcal{H}^s(E)=0\) when $s>\delta$.
	
	To show (2), we first define a mass distribution \(\mu\) on \(E\) as follows.
	For each interval \(I_{k,i}\) at the \(k\)-th iteration, set \(\mu(I_{k,i}) := \frac{1}{N_k} = \frac{1}{2^k}\); for any set \(A\subset\mathbb{R}\), if \(A\cap E_k=\emptyset\) for some \(k\in\mathbb{N}\), then we set \(\mu(A)=0\).
	By \cite[Proposition 1.7]{f14}, \(\mu\) can be extended to a measure on \(\mathbb{R}\) with \({\rm{spt}}\ \mu\subset E\) and
	\[
	\mu(E)=\lim_{k\to\infty}\mu(E_k)=1.
	\]
	Thus \(\mu\) is a mass distribution on \(E\). Moreover, for any \(k\in\mathbb{N}\) and any \(i=1,2,\dots,N_{k}\),
	\begin{align}\label{eq1230-1}
		\mu(I_{k,i})=\frac{1}{2^k}\le K(\text{diam } I_{k,i})^s,
	\end{align}
	where \[
	K:=\max_{k\in\mathbb{N}} 2^{-k\left(1-\frac{s}{\delta}\right)+\frac{ks}{\ln (k+1)}},
	\]
	which is finite due to \(s<\delta\).
	
	Let \(U \subset\mathbb{R}\) be any set with \(\rho:=\text{diam}\ U\in(0, 2^{-\frac{1}{\delta}-2}]\).
	Since \(L_k\) is strictly decreasing to \(0\) and \(L_1= 2^{-\frac{1}{\delta}-\frac 1{\ln 2}}>2^{-\frac{1}{\delta}-2}\), it follows that
	\begin{align}\label{eq316-1}
		(0, 2^{-\frac{1}{\delta}-2}]\subset \bigcup_{k\in\mathbb{N}}[L_{k+1}, L_{k}).
	\end{align}
	Furthermore, for any $k\in\mathbb{N}$,
	$$
	\frac{L_{k}}{L_{k+1}}=2^{\frac{1}{\delta}+\frac{k+1}{\ln (k+2)}-\frac{k}{\ln (k+1)}}\leq 2^{\frac{1}{\delta}+\frac{2}{\ln 3}-\frac{1}{\ln 2}}=: c_0.
	$$
	Hence, $[L_{k+1}, L_{k})\subset [\frac{1}{c_0}L_k, L_k)$.
	Combining this with \eqref{eq316-1}, we obtain
	$$
	(0, 2^{-\frac{1}{\delta}-2}]\subset \bigcup_{k\in\mathbb{N}}[L_{k+1}, L_{k})\subset\bigcup_{k\in\mathbb{N}}\left[\frac{1}{c_0}L_k, L_k\right).
	$$
	Consequently, there exists $k_0\in\mathbb{N}$ such that
	$\frac{1}{c_0}L_{k_0}\le\rho<L_{k_0}$.
	If \(U\) intersects with three or more intervals of the \(k_0\)-th iteration with $k_0\ge 2$,
	then \(\rho>L_{k_0}\), which is a contradiction.
	Hence \(U\cap E\) is covered by at most two intervals from the \(k_0\)-th iteration, denoted by \(I_{k_0,i_1},I_{k_0,i_2}\).
	By subadditivity and \eqref{eq1230-1},
	\[
	\mu(U)=\mu(U\cap E)\le\mu(I_{k_0,i_1})+\mu(I_{k_0,i_2})
	= \frac{1}{2^{k_0-1}}
	\leq 2K(L_{k_0})^s
	\leq 2Kc_0^s\rho^s.
	\]
	Therefore, from  Lemma \ref{them1230-1}, we deduce that
	\(\dim_{\mathcal{H}} (E) \ge s\). Thus, (2) holds true.
	This completes the proof of Proposition \ref{them1229-1}.
\end{proof}

\begin{proof}[Proof of Theorem \ref{lm-1226-x}]
	According to the proof of Proposition \ref{them1229-1}, for the given cube $Q_0:=[0,2)^n$, there exists a subset $E\subset Q_0$ such that
	$
	\dim_{\mathcal{H}} (E)=\delta$  and $\mathcal{H}^\delta(E) = 0.
	$
	Define
	\[
	w(x):=
	\begin{cases}
		\ 1, & x\in \mathbb{R}^n\setminus E,\\
		\ 0, & x\in E.
	\end{cases}
	\]
	Since the $\delta$-Hausdorff capacity and $\delta$-Hausdorff measure share the same null sets (see \cite[Lemma 1, p.64]{eg92}), it follows that $\ch_\fz^\delta(E)=0$ and hence $w(x)=1$ for $\mathcal{H}_\infty^\delta$-a.e.\ $x\in\mathbb{R}^n$
	which further implies $w\in\mathcal{A}_{p,\delta}$ for any $p\in [1,\fz)$.
	
	However, for any $\beta\in(0,\delta)$, we have $\mathcal{H}^\beta(E)=\infty$ by the definition of Hausdorff measure.
	Using again that the $\delta$-Hausdorff capacity and $\delta$-Hausdorff measure have the same null sets, we obtain $\mathcal{H}_\infty^\beta(E)>0$. This means that, for the above cube $Q_0$,
	$$\int_{Q_0}w(x)^{-\frac1{p-1}}\,d\ch_\fz^\beta\ge \int_E w(x)^{-\frac1{p-1}}\,d\ch_\fz^\beta=\fz,$$
	and hence $w\notin\mathcal{A}_{p,\beta}$ for any $p\in(1,\fz)$. Moreover, it is easy to know that
	$\cm_{\ch_\fz^\delta}w(x)>0$ on $E$ and, therefore, $w\notin\mathcal{A}_{1,\beta}$.
	This completes the proof of Theorem \ref{lm-1226-x}.
\end{proof}

\medskip



\medskip

\noindent Yangzhi Zhang and Ciqiang Zhuo (Corresponding author)

\smallskip

\noindent MOE-LCSM, School of Mathematics and Statistics,
Hunan Normal University,
Changsha, Hunan 410081, P. R. China

\smallskip

\noindent{\it E-mails:}
\texttt{yzzhang@hunnu.edu.cn}

\noindent\phantom{ {\it E-mails}}
\texttt{cqzhuo87@hunnu.edu.cn}

\end{document}